\documentclass[12pt]{amsart}
\usepackage{amsmath}
\usepackage{amssymb}
\usepackage{graphicx}
\usepackage{a4wide}

\begin{document}
\title{Lower bounds for the simplexity of the n-cube}
\author{Alexey Glazyrin}
\address{University of Texas at Brownsville, 80 Fort
Brown, Brownsville, TX, 78520, USA}

\email{Alexey.Glazyrin@gmail.com}

\maketitle

\begin{abstract}
In this paper we prove a new asymptotic lower bound for the
minimal number of simplices in simplicial dissections of
$n$-dimensional cubes. In particular we show that the number of simplices in dissections of $n$-cubes without additional vertices is at least $(n+1)^{\frac {n-1} 2}$.
\end{abstract}

\pagestyle{plain} \hyphenation{} \textwidth=13cm \textheight=20cm

\newtheorem{lem}{Lemma}
\newtheorem{thm}{Theorem}
\newtheorem{cor}{Corollary}
\newtheorem{sta}{Statement}
\newtheorem{prop}{Proposition}
\theoremstyle{definition}
\newtheorem{defn}{Definition}
\theoremstyle{remark}
\newtheorem*{rem}{Remark}
\newtheorem{expl}{Example}
\newcommand { \ib }[1] {\textit{\textbf{#1}}}
\newcommand{\R}{\ensuremath{\mathbb{R}}}

\renewcommand{\proofname}{Proof}
\makeatletter \headsep 10 mm \footskip 10 mm
\renewcommand{\@evenhead}%

\vskip 15pt

\section{Introduction}
This work is devoted to some properties of dissections of convex
polytopes into simplices with vertices in vertices of the
polytope. From now on by a dissection we mean representation of a
polytope as a union of non-overlapping (i.e. their interiors do
not intersect) simplices. In case each two simplices of a
dissection intersect by their common face we will call such
dissection a triangulation. Obviously each
dissection of a convex planar polygon is a triangulation. However for higher-dimensional convex polytopes
that is not true.

One of the most important problems concerning triangulations is that of finding the minimal triangulation for a given polytope,
i.e. triangulation with the minimal number of simplices. For a
polygon the number of triangles in a triangulation is always equal
to $v-2$, where $v$ is the number of vertices of the polygon. The
situation is very different even for the three-dimensional case.
Three-dimensional cubes can be triangulated both into six or into five
tetrahedra.

In the next section of this paper we consider dissections of
\textit{prismoids} and prove some properties for them (see also \cite{Glaz}). By
prismoids we mean $n$-dimensional polytopes all vertices of which
lie in two parallel $(n-1)$-dimensional hyperplanes. For instance,
the set of prismoids contains cubes, prisms, 0/1-polytopes (i.e.
polytopes all Cartesian coordinates of which are 0 or 1, see
\cite{Ziegler2, Ziegler}). In the third section we will show how these
properties can be used for finding lower bounds for the number of simplices in simplicial dissections
of the $n$-dimensional cube. In the last section we will prove the new
asymptotic lower bound for the number of simplices in simplicial dissections of the $n$-cubes.

We use the following notations: $dis(n)$ is the minimal number of
simplices in a dissection of the $n$-dimensional cube, $triang(n)$
is the minimal number of simplices in a triangulation of the
$n$-dimensional cube, and $\rho(n)$ is the maximal determinant of a
$0/1$-matrix. The term simplexity used in the literature is somewhat ambiguous. We will use this term for $triang(n)$. Obviously $triang(n) \geq dis(n)$. In our work all the lower bounds will be given for dissections and subsequently they are all
true for the simplexity.

There is an obvious lower bound for $dis(n)$:$$dis(n)\geq \frac
{n!} {\rho(n)}.$$ The maximal volume of a simplex with vertices in
the vertices of the $0/1$-cube is not greater than $\frac
{\rho(n)} {n!}$, therefore we immediately achieve this bound. An
upper bound for $\rho(n)$ can be easily obtained by some matrix
transformations and Hadamard's inequality (see \cite{Ziegler2} for details; the generalization of this
inequality will be proved in the last section of the paper):
\begin{lem}\label{lem:haddet}
$\rho(n)\leq 2 \left(\dfrac {\sqrt{n+1}} {2}\right)^{n+1}$
\end{lem}
Hence the following bound is true
\begin{thm}\label{thm:euclbound}
$$dis(n)\geq \dfrac {n!} {2 \left(\dfrac {\sqrt{n+1}}
{2}\right)^{n+1}} =: E(n)$$
\end{thm}

Better bounds can be achieved by using other volumes instead of
the Euclidean volume. The following bound was proved by W.D~Smith
in \cite{Smith} by means of hyperbolic volume and was the best
asymptotic bound up to the moment.

$$dis(n) \geq H(n) \geq \frac 1 2 6^{\frac n 2} (n+1)^{- \frac {n+1} 2} n!$$

$$\lim_{n \rightarrow \infty} \left(\frac
{H(n)} {E(n)} \right)^{\frac 1 n} = A > 1.2615$$

In Table \ref{table:known_bounds} lower bounds for minimal numbers of simplices
in triangulations and dissections are shown up to dimension eight (for previous resulsts see \cite{Mara,Sallee2,Todd}).
The sign ``='' is used when the number is known exactly.

\begin{table}[h]
\caption{Known bounds}
\label{table:known_bounds}
\centering
\begin{tabular}{|c|c|c|}
    \hline
    n & triang(n) & dis(n) \\
    \hline
    3 & 5 & 5 \\
    4 & =16 \cite{Cottle, Sallee} & =16 \cite{Hughes} \\
    5 & =67 \cite{HA} & 61 \cite{Hughes} \\
    6 & =308 \cite{HA} & 270 \cite{HA}\\
    7 & =1493 \cite{HA} & 1175 \cite{HA}\\
    8 & 5522 \cite{Hughes} & 5522 \cite{Hughes} \\
    \hline
  \end{tabular}
\end{table}

One can also consider triangulations and dissections with additional vertices. Some bounds for simplicial covers and triangulations with additional vertices were obtained in~\cite{bliss-su}.

Smith's method \cite{Smith} is convenient for dissections with additional vertices. We deal only with triangulations and dissections with vertices in vertices of a cube. Hence our result is not a total improvement of the bounds achieved by Smith. There exist well-known examples where the use of new points make a dramatic difference on the size of the dissections \cite{brehm}. It is a famous open problem for many years whether the $n$-cube allows such phenomena.

Upper bounds for $triang(n)$ can be obtained by constructing explicit examples \cite{haiman}. The best bound for the moment is $O(0.816^n n!)$ \cite{OS}.

A quite extensive survey on the minimal simplicial dissections and triangulations of n-cubes can be found in the papers \cite{bliss-su, Smith} mentioned above (see also \cite{zong1,zong2}).

\section{Triangulation of prismoids}

Let all the vertices of an $n$-dimensional polytope $P$ in $\mathbb{R}^n$ lie on two parallel $(n-1)$-dimensional hyperplanes,
i.e. $P$ is an $n$-dimensional prismoid. Without loss of generality we can consider hyperplanes $x_1=0$ and $x_1=1$ (none of the following statements depend on the distance between hyperplanes). Assume also that we have a dissection $\Delta$ of the polytope $P$ into $n$-dimensional simplices. All the vertices of the simplices are vertices of the polytope.

Denote by $S_i$ the set of all simplices with $i$ vertices of $P$ in $x_1=0$ and $(n+1-i)$ vertices of $P$ in $x_1=1$.

Define $\Delta_i=\{T\in\Delta|$ exactly $i$ vertices of $T$ lie in $x_1=0\}$. So $\Delta_i=S_i\bigcap \Delta$. Denote by $q_i$ the
number of simplices in $\Delta_i$, by $T_i^j$ -- the $j$-th simplex in $\Delta_i$, and by $V(T_i^j)$ its $n$-dimensional Euclidean volume (further on all $n$-dimensional Euclidean volumes will be denoted by $V(\cdot)$, and all $(n-1)$-dimensional Euclidean volumes will be denoted by $S(\cdot)$).

\begin{thm}\label{thm:prismoid}
For the prismoid $P$ and its simplicial dissection $\Delta$ let
$\ V(i)$ be the total volume of simplices in $\Delta_i$, i.e.
$V(i)=\sum \limits_{j=1}^{q_i}V(T_i^j)$. Then for each $i$ such that $1\leq i \leq n$, $V(i)$ depends only on $P$, the choice of parallel hyperplanes containing all the vertices of $P$, and $i$, and does not depend on $\Delta$.
\end{thm}

\begin{rem}
$P$ can be a prismoid with respect to two different pairs of
parallel hyperplanes. In this theorem we mean that the pair is
fixed.
\end{rem}

For proving Theorem \ref{thm:prismoid} we need several lemmas.

\begin{lem}\label{lem:mink}
Consider $T\in S_i$ and its intersection $M_t$ with a hyperplane
$x_1=t$, where $t\in[0,1]$. $(n-1)$-dimensional volume $S(M_t)$ satisfies equality $S(M_t)=c(1-t)^{i-1} t^{n-i}$,
where $c$ is some constant not depending on $t$.
\end{lem}

\begin{proof}
Let $A$ be a convex hull of $i$ vertices of the simplex $T$ from the hyperplane $x_1=0$ and let $B$ be a convex hull of
$(n+1-i)$ vertices of $T$ from the hyperplane $x_1=1$. We will show now that $M_t=\{(1-t)A+tB\}$ (here by $\{+\}$ we mean
Minkowski sum of these two sets). Note that any point $Z$ of the intersection we consider divides some segment $XY$ with ratio
$t:(1-t)$, where $X\in A$ and $Y\in B$. Thus $Z=(1-t)X+tY$ and it is obvious that all the points $Z$ that can be expressed this way
lie in the intersection $M_t$.

Let $A_j$ for $1\leq j \leq i$ be the $j$-th vertex of the simplex $A$ and let $B_k, 1\leq k \leq n+1-i,$ be the $k$-th vertex of the
simplex $B$. Notice that all the vectors $\overrightarrow{A_1A_j}$ ($1<j\leq i$) and $\overrightarrow{B_1B_k}$ ($1<k\leq n+1-i$) are linearly independent (over $\R$) altogether (otherwise vectors $\overrightarrow{A_1A_j}$, $\overrightarrow{A_1B_1}$, $\overrightarrow{B_1B_k}$ are linearly dependent and consequently $\overrightarrow{A_1A_j}$, $\overrightarrow{A_1B_k}$ are linearly dependent which contradicts the fact that $A_j$, $B_k$ are the vertices of the $n-$dimensional simplex). Let $O$ be a point of intersection of $A_1B_1$ and the hyperplane $x_1=t$. Now we scale $M_t$ about $O$ with a coefficient $\frac 1 {1-t}$ along the vectors $\overrightarrow{A_1A_j}$ and with a coefficient $\frac 1 t$ along the vectors $\overrightarrow{B_1B_k}$. After this transformation $M_t$ will change to a figure congruent to $\{A+B\}$. Because of the linear independence of $\overrightarrow{A_1A_j}$, $\overrightarrow{B_1B_k}$ we achieve that $S(M_t)=t^{n-i}(1-t)^{i-1}S(A+B)$, where $S$ is an
$(n-1)$-dimensional Euclidean volume. We take $S(A+B)$ as $c$ and the lemma is proved.
\end{proof}

We will use the following lemma:
\begin{lem}\label{lem:bernind}
For each $m\in\mathbb{N}$ polynomials $Q_i=t^i(1-t)^{m-i}$ where $0\leq
i\leq m$ (Bernstein basis polynomials \cite{Bern}) are linearly
independent over $\R$.
\end{lem}

Now consider any simplicial dissection $\Delta$ of the polytope
$P$. For each simplex $T_i^j\in \Delta$ denote the correspondent constant from Lemma \ref{lem:mink} by $c_i^j$. Define $c_i(\Delta)=\sum\limits_{j=1}^{q_i} c_i^j(\Delta)$.

\begin{lem}\label{lem:coefind}
All $c_i$ are independent of $\Delta$ and are determined only by the
polytope $P$, the choice of parallel hyperplanes containing all the vertices of $P$, and $i$.
\end{lem}
\begin{proof}
Suppose we have two dissections $\Delta$ and $\Delta'$. Let us
prove that $c_i(\Delta)=c_i(\Delta')$ for all $i$ such that $0 \leq i \leq
n$. If $S(t)$ is the $(n-1)$-dimensional volume of intersection
of a hyperplane $x_1=t$ with $P$, then we achieve that
$c_1(\Delta)Q_{n-1}+\ldots+c_n(\Delta)Q_{0}= S$ (here
$Q_i$ are Bernstein polynomials for $m=n-1$ and the equality is for functions on $[0,1]$). Analogously
$c_1(\Delta')Q_{n-1}+\ldots+c_n(\Delta')Q_{0}= S$. Hence
$(c_1(\Delta)-c_1(\Delta'))Q_{n-1}+\ldots+(c_n(\Delta)-c_n(\Delta'))Q_{0}=
0$. By Lemma \ref{lem:bernind}, $Q_0,\ldots,Q_{n-1}$ are linearly
independent. Thus $c_i(\Delta)=c_i(\Delta')$.
\end{proof}

\begin{proof}[Proof of Theorem 2]
Express $V(T_i^j)$ in terms of $c_i^j$. Using the formula for the volume of an
$(n-1)$-dimensional section of the simplex by a hyperplane $x_1=t$
we obtain \footnote{Here $B$ and $\Gamma$ are standard Euler
functions: $B(x,y)=\int_0^1 t^{x-1}(1-t)^{y-1}dt$, $\Gamma(x)=\int_0^{\infty}e^{-t}t^{x-1}dt$}
$$V(T_i^j)=\int \limits_0^1c_i^jt^{n-i}(1-t)^{i-1} dt=
c_i^jB(n-i+1,i)=$$
$$=c_i^j\dfrac {\Gamma(n-i+1)\Gamma(i)} {\Gamma(n+1)}=c_i^j\dfrac
{(n-i)!(i-1)!} {n!}=\dfrac {c_i^j} {n\left({n-1}\atop
{i-1}\right)}.$$

Thus we achieve that $$\sum \limits_{j=1}^{q_i}V(T_i^j)=\dfrac
{\sum \limits_{j=1}^{q_i}c_i^j} {n\left({n-1}\atop
{i-1}\right)}=\dfrac {c_i} {n\left({n-1}\atop {i-1}\right)}.$$
Denote the right part of this equation by $V(i)$ and the theorem
is proved.
\end{proof}

\begin{cor}\label{cor:prism}
Denote by $S(t)$ the $(n-1)$-dimensional volume of a section of $P$ by a hyperplane
$x_1=t$.
If all conditions of Theorem \ref{thm:prismoid} hold and
$S(t)$ is constant then $V(i)=\frac 1 n V(P)$.
\end{cor}

\begin{proof}
Suppose $S(t)= S_0$. Then $c_1Q_{n-1}+\ldots+c_nQ_{0}=
S_0$. Notice that if $\beta_i=S_0\left({n-1}\atop {i-1}\right)$
then $\beta_1Q_{n-1}+\ldots+\beta_nQ_{0}=S_0\left({n-1}\atop
{0}\right)t^0(1-t)^{n-1}+\ldots+S_0\left({n-1}\atop
{n-1}\right)t^{n-1}(1-t)^0= S_0$. Similarly to the proof of
Theorem \ref{thm:prismoid} we get $(\beta_1-c_1)Q_{n-1}+\ldots+(\beta_n-c_n)Q_{0}=0$ and using the linear
independence of $Q_i$, $i=0,1,\ldots,n-1,$ we obtain that
$c_i=\beta_i$.
$$V(i)=\dfrac {c_i} {n\left({n-1}\atop {i-1}\right)}=\dfrac
{\beta_i} {n\left({n-1}\atop {i-1}\right)}= \dfrac {S_0} n.$$
Since $V(P)=\int \limits_0^1 S(t) dt=S_0$, the
corollary is proved.
\end{proof}
Notice that this corollary works for prisms and particularly for
cubes. We will use that for the following section.

\section{Lower bounds for the simplexity of cubes}

\subsection{The general construction for the lower bound of the simplexity of the $n$-cube.}

Consider any $n$-dimensional $0/1$-simplex $T$ and suppose that
its vertices $A_1,\ldots,A_{n+1}$ have coordinates
$A_1(a_{1,1},\ldots,a_{1,n}),\ldots,A_{n+1}(a_{n+1,1},\ldots,a_{n+1,n})$.
Notice that the $n-$dimensional Euclidean volume of $T$ is equal to $\frac 1 {n!}$
multiplied by the absolute
value of the determinant of the following matrix: $$M = \left(%
\begin{array}{cccc}
  1 & a_{1,1} & \ldots & a_{1,n} \\
  1 & a_{2,1} & \ldots & a_{2,n} \\
  \vdots & \vdots & \ddots & \vdots \\
  1 & a_{n+1,1} & \ldots & a_{n+1,n} \\
\end{array}%
\right)$$

Denote the $j$-th column of this matrix by $b_{j-1}$ (for instance,
$b_0$ is the first column consisting of $(n+1)$ 1's) and
$\|b_j\|^2$ by $h_j$ (here we mean the Euclidean norm, i.e. $h_j$ is
just a number of 1's in the $j$-th column). Then we define functions
$V_{k,m}(T) = V(T)$ if $h_k = m$ and $V_{k,m}(T) = 0$ if $h_k \neq
m$.

Considering the unit cube as the prism with vertices in two parallel hyperplanes $x_k=0$ and $x_k=1$ and applying Corollary \ref{cor:prism} for $i=m$ we obtain:
\begin{prop}
For each dissection $\Delta$ of the $n$-dimensional cube and for
all $k$ and $m$ such that $1 \leq k,m \leq n$ $$ \sum \limits_{T \in \Delta} V_{k,m}(T)
= \frac 1 n.$$
\end{prop}

Now take any $n \times n$ matrix of coefficients $\alpha_{k,m}$ such that $\sum \limits_{1 \leq k,m \leq n} \alpha_{k,m} = n$. Then by the proposition we have
$$\sum \limits_{T \in \Delta} \sum \limits_{1 \leq k,m \leq n} \alpha_{k,m} V_{k,m}(T) =  \sum \limits_{1 \leq k,m \leq n} \alpha_{k,m} \sum \limits_{T \in \Delta} V_{k,m}(T) =$$
$$=\sum \limits_{1 \leq k,m \leq n} \frac {\alpha_{k,m}} n = 1$$

Define $V^{\alpha}(T) = \sum \limits_{1 \leq k,m \leq n}
\alpha_{k,m} V_{k,m}(T)$. We will also use the term \emph{$\alpha$-volumes} for $V^{\alpha}(T)$. Then $\sum \limits_{T \in \Delta}
V^{\alpha}(T) = 1$  for any dissection $\Delta$ and, therefore, $dis(n) \geq \frac 1 {\max \limits_T
V^{\alpha}(T)}$. So in order to get the best bound we must find
$$G = \min \limits_{\alpha} \max \limits_T V^{\alpha}(T),$$
which is a problem of linear programming with respect to $\alpha$.

The following formula can simplify calculations for weighted volumes.
\begin{prop}\label{prop:alpha}
$V^{\alpha}(T) = V(T)\sum\limits_{j=1}^n\alpha_{j, h_j}$
\end{prop}
 
\begin{proof}
$V^{\alpha}(T) =  \sum \limits_{1 \leq k,m \leq n}
\alpha_{k,m} V_{k,m}(T) =  \sum \limits_{k=1}^n \sum \limits_{m=1}^n 
\alpha_{k,m} V_{k,m}(T)$

Among $V_{k,m}(T)$  ($1\leq m \leq n$, $k$ is fixed) there is only one non-zero value and that is $V_{k, h_k} (T) = V(T)$. Therefore $\sum \limits_{m=1}^n \alpha_{k,m} V_{k,m}(T) = \alpha_{k, h_k} V(T)$ and the formula is proved.
\end{proof}

We will call $\sum\limits_{j=1}^n\alpha_{j, h_j}$ the \emph{weight part} of $V^{\alpha}(T)$.

We can simplify our linear program. Assume $G$ is attained on some matrix $\alpha^1$. Consider all matrices $\alpha^2,\ldots,\alpha^t$ that can be obtained from $\alpha$ by compositions of permutations of rows ($\alpha_{k_1,m}$ $\leftrightarrow$ $\alpha_{k_2,m}$ for all $m$ such that $1\leq m\leq n$) and reflections of rows ($\alpha_{k,m}$ $\leftrightarrow$ $\alpha_{k,n+1-m}$ for all $m$ such that $1\leq m\leq n$). These transformations naturally represent cube symmetries. Then $\alpha=\frac {\alpha^1+\alpha^2+\ldots+\alpha^t} {t}$ is symmetric with respect to these transformations and satisfies the condition $\sum \limits_{1 \leq k,m \leq n} \alpha_{k,m} = n$. From Proposition \ref{prop:alpha} for each $T$, $V^{\alpha}(T)=\frac {V^{\alpha^1}(T)+V^{\alpha^2}(T)+\ldots + V^{\alpha^t}(T)} {t}$, and $\max \limits_T V^{\alpha}(T) = \max \limits_T \frac {V^{\alpha^1}(T)+\ldots + V^{\alpha^t}(T)} {t}\leq \frac {\max \limits_T V^{\alpha^1}(T) + \ldots + \max \limits_T V^{\alpha^t}(T)} {t} = G$ and, therefore, $\max \limits_T V^{\alpha}(T) = G$. Hence it is enough to consider $\alpha$ satisfying $\alpha_{k_1,m} = \alpha_{k_2,m}$ and $\alpha_{k,m} = \alpha_{k,n+1-m}$ for all $m$, $1 \leq m \leq n$. From now on we use notations $\alpha_m =
\alpha_{k,m}$ with the conditions $\sum\limits_{m=1}^n \alpha_m = 1$ and
$\alpha_m = \alpha_{n+1-m}$ for all $m$.

The other idea that can help us to simplify the linear program is the following. If we have simplices with the same weight parts of $V^{\alpha}$ but different Euclidean volumes we can consider only those with the largest Euclidean volume for the purpose of finding the maximum. The weight part may be negative but then these simplices definitely are not representatives of the largest $V^{\alpha}$ since the sum of $\alpha$-volumes is exactly 1 for any dissection and subsequently some positive $\alpha$-volumes exist.

Taking into account these simplifications, we found some lower bounds for small dimensions.

\begin{expl}[Simplexity of the 4-cube]
Using the simplifying conditions on $\alpha$ we have $\alpha_1+\alpha_2+\alpha_3+\alpha_4=1$ and $\alpha_1=\alpha_4$, $\alpha_2=\alpha_3$. Hence we get $\alpha_1+\alpha_2=\frac 1  2$. From Proposition \ref{prop:alpha} the weight parts of $V^{\alpha}$ may be only $4\alpha_1$, $3\alpha_1+\alpha_2$, $2\alpha_1+2\alpha_2$, $\alpha_1+3\alpha_2$, $4\alpha_2$. In the first four cases the distribution of simplex vertices for some pair of parallel facets of the cube is 1-4 (one vertex on one facet and four other vertices on the parallel facet). So this simplex can be considered as a pyramid with its base in one of the cube facets. As we know, a simplex with vertices in the vertices of the $3$-dimensional $0/1$-cube has volume equal to $\frac 1 3$ if the distribution of its vertices for each pair of parallel facets of this $3$-cube is 2-2 and volume equal to $\frac 1 6$ in all other cases. The first case for the base of the pyramid gives the $4$-dimensional volume equal to $\frac 1 4 \cdot \frac 1 3 = \frac 2 {24}$ and corresponds to the weight volume $\alpha_1+3\alpha_2$. Other cases for the base give the $4$-dimensional volume equal to $\frac 1 4 \cdot \frac 1 6 = \frac 1 {24}$ and correspond to the weight volumes $4\alpha_1$, $3\alpha_1+\alpha_2$, $2\alpha_1+2\alpha_2$. From Lemma \ref{lem:haddet}, the volume of any $4$-dimensional $0/1$-simplex is not greater than $\frac 3 {24}$ and it is attained on the simplex with vertices $(0,0,0,0)$, $(0,1,1,1)$, $(1,0,1,1)$, $(1,1,0,1)$, $(1,1,1,0)$. On the whole we showed that the maximum Euclidean volumes for weight volumes $4\alpha_1$, $3\alpha_1+\alpha_2$, $2\alpha_1+2\alpha_2$, $\alpha_1+3\alpha_2$, $4\alpha_2$ are $\frac 1 {24}, \frac 1 {24}, \frac 1 {24}, \frac 2 {24}, \frac 3 {24}$ respectively. Therefore, we need to minimize $\max\{\frac {4\alpha_1} {24}, \frac {3\alpha_1+\alpha_2} {24}, \frac {2\alpha_1+2\alpha_2} {24}, \frac {2\alpha_1+6\alpha_2} {24}, \frac {12\alpha_2} {24}\}$ given $\alpha_1+\alpha_2 = \frac 1 2$. In terms of the linear programming we maximize $-m$ given $m \geq \frac {4\alpha_1} {24}, m \geq \frac {3\alpha_1+\alpha_2} {24}, m \geq \frac {2\alpha_1+2\alpha_2} {24}, m \geq \frac {2\alpha_1+6\alpha_2} {24}, m \geq \frac {12\alpha_2} {24}, \alpha_1+\alpha_2 = \frac 1 2.$ Though the feasible region is unbounded, it is bounded in the direction of the gradient of the objective function and hence the optimal value is attained. Solving this linear program we obtain the optimal value $\min \limits_{\alpha} \max \limits_T V^{\alpha}(T) = \frac 1 {16}$ for $\alpha_1=\frac 3 8$ and $\alpha_2=\frac 1 8$. Subsequently the minimal dissection contains at least $16$ simplices.

The condition $V^{\alpha}(T) = \frac 1 {16}$ for $\alpha_1=\frac 3 8$ and $\alpha_2=\frac 1 8$ is quite restrictive. Only simplices with the weight parts $4\alpha_1$, $\alpha_1+3\alpha_2$, $4\alpha_2$ and the maximal possible Euclidean volume for this part satisfy it. Denote the number of simplices in these groups by $m_1, m_2$ and $m_3$ respectively. Simplices from the first group are corner simplices. Since corner vertices of corner simplices in a dissection cannot be adjacent, $m_1\leq 8$. Using the additional equation on the Euclidean volume we obtain that
$$1=\frac 1 {24} m_1 + \frac 2 {24} m_2 + \frac 3 {24} m_3$$
$$16=m_1+m_2+m_3$$
$$m_1\leq 8$$

The only non-negative integer solution of this system is $m_1=8, m_2=8, m_3=0$ and subsequently we obtain another proof for the structure of minimal triangulations of the $4$-cube (see \cite{Cottle}).
\end{expl}

Analogously by the exhaustive case analysis for $n = 5,6$ we were able to
find all linear constraints but the lower bounds, 60 and 240 respectively, obtained by our linear program are
smaller than the known bounds for the number of simplices in
dissections. Nevertheless this method allows us to prove the new
asymptotic lower bound.

\subsection{New asymptotic lower bound}

Let us prove a generalization of Lemma \ref{lem:haddet}. Here we use the same notations as in the previous subsection.
\begin{lem}\label{lem:detgen}
For any $n$-dimensional $0/1$-simplex $T$,
$(det\ M)^2 \leq (n+1)^{1-n} \prod \limits_{j=1}^n h_j(n+1 - h_j)$.
\end{lem}

\begin{proof}
For each column $b_j$ of $M$, $j \geq 1,$ make a transformation $\phi_j: b_j \longmapsto
b_j'= b_j - \frac {h_j} {n+1} b_0$. Each of these transformations does not change
$det\ M$. Then every 1 in $b_j$ will be changed to $\frac {n+1-h_j} {n+1}$ and every 0 will be changed to $-\frac {h_j} {n+1}$. Therefore, $$\|b_j'\|^2 = h_j \frac {(n+1-h_j)^2}
{(n+1)^2} + (n+1-h_j) \frac {h_j^2} {(n+1)^2} = \frac {h_j
(n+1-h_j)} {n+1}.$$ By Hadamard's inequality $$(det\ M)^2 \leq
(n+1) \prod \limits_{j=1}^n \|b_j\|^2 = (n+1)^{1-n}
\prod \limits_{j=1}^n h_j(n+1 -
h_j).$$ Hence the lemma is proved.
\end{proof}

Now we explicitly set $\alpha_i = \frac 1 n (1 + \ln(n!)) - \frac 1 2 \ln(i(n+1-i))$ for all $i$, $1\leq i \leq n$, so that $\sum \alpha_i = 1$ and $\alpha_i=\alpha_{n+1-i}$. Using Proposition \ref{prop:alpha} and the formula $V(T)=\frac 1 {n!}|det M|$ we get

$$V^{\alpha}(T)= \frac 1 {n!} |det M| \sum
\limits_{j=1}^{n}\alpha_{h_j} = \frac 1 {n!} |det M| \sum
\limits_{j=1}^{n}\left(\frac 1 n (1 + \ln(n!)) - \frac 1 2 \ln(h_j(n+1-h_j))\right) =$$ $$=\frac 1
{n!} |det M| \left(1 + \ln(n!) - \frac 1 2 \ln\prod \limits_{j=1}^n
h_j (n+1 - h_j)\right)$$ Then we use the
inequality on $det\ M$ from Lemma \ref{lem:detgen} (we are
interested only in positive $\alpha$-volumes, for them the use of the inequality is correct):
$$V^{\alpha}(T) \leq \frac 1 {n!} \sqrt{(n+1)^{1-n} \prod
\limits_{j=1}^n h_j (n+1 - h_j)} \left(1 +
\ln(n!) - \frac 1 2 \ln\prod \limits_{j=1}^n h_j (n+1 - h_j)\right).$$ Denote $\frac 1 2 \ln\prod
\limits_{j=1}^n h_j (n+1 - h_j)$ by $t$.
Then this inequality can be rewritten in the following
form:$$V^{\alpha}(T) \leq \frac 1 {n!} (n+1)^{\frac {1-n} 2}
g(t),$$ where $g(t) = e^t (1 + \ln(n!) - t)$. Let us find the maximum of this function. $g'(t) = e^t (\ln(n!) - t)$, so $g(t)$
reaches its maximum at $t = \ln(n!)$ and $\max{g(t)} = n!$ Thus,
$$V^{\alpha}(T) \leq \frac 1 {n!} (n+1)^{\frac {1-n} 2} n! =
(n+1)^{\frac {1-n} 2},$$ and we have proved the following theorem.

\begin{thm}\label{thm:asbound}
For any natural $n$ $$dis(n) \geq (n+1)^{\frac {n-1} 2} =: F(n)$$
\end{thm}

This bound gives the asymptotic improvement with respect to
the Euclidean bound $$\lim_{n \rightarrow \infty} \left(\frac
{F(n)} {E(n)} \right)^{\frac 1 n} = \lim_{n \rightarrow \infty}
\left(\frac {n^{\frac n 2}} {n^{\frac n 2} (\frac 2 e)^n}
\right)^{\frac 1 n} = \frac e 2 \approx 1.359140914$$

\section{Acknowledgements}
For useful discussions and valuable comments the author thanks
Arseniy Akopyan, Nikolai Dolbilin, Alexey Garber, and G\"unter M.
Ziegler.

This work is supported by the Russian government project 11.G34.31.0053.

\end{document}